\documentclass[11pt,leqno]{amsart}

\usepackage{amsmath,amssymb,amsthm}

\newcommand{\dis}{\displaystyle}

 \newcommand{\Om}{\Omega}

\newcommand{\la}{\lambda}

\newtheorem{theo}{Theorem}[section]

\newtheorem{lem}[theo]{Lemma}

\newtheorem*{example}{Example}

\def\r{\mathbb{R}}
\def\rn{\mathbb{R}^N}
\def\rm{\mathbb{R}^{M}}
\def\eps{\varepsilon}

\def\irm{\int_{\rm}}
\def\io{\int_{\Omega}}
\def\iom{\int_{\omega}}
\def\wt{\widetilde}
\def\calD{\mathcal{D}}

\numberwithin{equation}{section}

\title[Weakly coercive problems]{On some weakly coercive quasilinear problems\\ with forcing}

\author{Andrzej Szulkin} 
\address{Department of Mathematics, Stockholm University, 106 91  Stockholm, Sweden}
\email{andrzejs@math.su.se}

\author{Michel Willem} 
\address{Inst.\ de Recherche  en Math.\ et Phys., Universit\'e catholique de Louvain, 1348 Louvain-la-Neuve, Belgium}
\email{michel.willem@uclouvain.be}



\begin{document}

\baselineskip15pt

\maketitle

\begin{abstract}
We consider the forced problem
$-\Delta_p u - V(x)|u|^{p-2} u = f(x)$,
where $\Delta_p$ is the $p$-Laplacian ($1<p<\infty$) in a domain $\Omega\subset \mathbb{R}^N$,  $V\ge 0$ and
$Q_V (u) :=  \int_\Omega |\nabla u|^p\, dx - \int_\Omega V|u|^p\,dx$ satisfies the condition (A) below. 
We show that this problem has a solution for all $f$ in a suitable space of distributions. Then we apply this result to some classes of functions $V$ which in particular include the Hardy potential  \eqref{hardy2} and the potential $V(x)=\lambda_{1,p}(\Omega)$, where $\lambda_{1,p}(\Omega)$ is the Poincar\'e constant on an infinite strip. 
\end{abstract}

\section{Introduction} \label{intro}

Our purpose is to solve the forced problem
\[
-\Delta_p u - V(x)|u|^{p-2} u = f(x)
\]
where $\Delta_p u := \textrm{div}\  (|\nabla u|^{p-2} \nabla u)$ and $f\in \mathcal{D}' (\Om)$, the space of distributions on the domain~$\Om$.

\medskip

Our assumptions are the following:\\[0.2cm]
\textbf{(A)}
\vspace*{-4.5mm}

\hspace*{5mm}
\begin{minipage}{140mm} \emph{ $1<p<\infty$, $V\in L^r_{\emph{loc}} \ (\Om)$ with $r$ as in \eqref{eq1.2}, $\Om$ is a domain in $\mathbb{R}^N$, $V \geq 0$ and for all test functions $u\in \mathcal{D} (\Om) \backslash \{0\}$,
\begin{equation} \label{eq1.1}
\mathcal{Q}_V (u) :=  \int_\Om |\nabla u|^p \,dx - \int_\Om V|u|^p \,dx > 0.
\end{equation}
There exists $1<q\leq p$ such that $p-1<q$,
\begin{equation} \label{eq1.2}
r=1 \ (N<q), \quad 1<r<+\infty \  (N=q), \quad 1/r+(p-1)/q^\ast=1 \  (N>q)
\end{equation}
and there exists $W \in \mathcal{C} (\Om)$, $W>0$, such that for all $u\in\mathcal{D}(\Om)$,
\begin{equation} \label{eq1.3}
\dis \left(\int_\Om(|\nabla u|^q + |u|^q)W\,dx\right)^{p/q} \leq \mathcal{Q}_V (u).
\end{equation}
}
\end{minipage}

\medskip

\par\noindent
Let us recall that $q^\ast:=Nq/(N-q)$.

\medskip

Our first example of $V$ is the \emph{quadratic Hardy potential} $(N\geq 3$, $p=2)$:
\begin{equation} \label{hardy1}
V(x):=\left({N-2\over 2}\right)^2 |x|^{-2}.
\end{equation}
The corresponding forced problem is solved in \cite{6} using the Brezis-Vazquez remainder term for the quadratic Hardy inequality (\cite{4} and \cite{10}).

A second example is the \emph{Hardy potential} $(1<p<N)$:
\begin{equation} \label{hardy2}
V(x) := \left({N-p\over p}\right)^p |x|^{-p}.
\end{equation}
The corresponding forced problem is partially solved in \cite{3} using the Abdellaoui-Colorado-Peral remainder term for the Hardy inequality \cite{1}.

A third example is a potential $V \in L^\infty_{\textrm{loc}}(\Om)$ such that (1.3) is satisfied with $p=q$ (see \cite{9}).

If $p=2$, the natural energy space is the completion $H$ of $\mathcal{D}(\Om)$ with respect to the norm $[\mathcal{Q}_V(u)]^{1/2}$, and it suffices to have $V\in L^1_{\textrm{loc}}(\Om)$. Then $H$ is a Hilbert space with an obvious inner product. The following result is immediate:

\begin{theo} \label{thma}
For each $f\in H^*$ (the dual space of $H$) the problem
\[
-\Delta u - V(x)u = f(x)
\]
has a unique solution $u\in H$. 
\end{theo} 

This is an extension of Lemma 1.1$'$ in \cite{6} though the argument is exactly the same as there. As we shall see at the beginning of Section \ref{appl}, more can be said about $u$ and $H^*$ if $V$ is the Hardy potential \eqref{hardy1} in a bounded domain.  

When $p\neq 2$, one can expect  no uniqueness as in Theorem \ref{thma}, see \cite{dem}, pp.\ 11-12 and \cite{fhtt}, Section 4. Hence the functional $[\mathcal{Q}_V(u)]^{1/p}$ is no longer convex, so it cannot serve as a norm,  and the second conjugate functional $Q^{\ast\ast}_V$ was used in \cite{9} to define the energy space. 

Our goal is to generalize all the above results by using only the Hahn-Banach theorem to define an energy space and to obtain a priori bounds.

We also consider the case of constant potential
\[
V\equiv \la_{1, p}(\Om):=\inf \left\{\int_\Om |\nabla u |^p \,dx : u\in \mathcal{D}(\Om),\int_\Om |u|^p \,dx=1\right\}
\]
on the cylindrical domain $\Om = \omega \times \mathbb{R}^M$ where $\omega\subset\r^{N-M}$ is bounded. Then we have
\begin{equation} \label{pi}
\la_{1,p} (\Om)=\la_{1,p}(\omega) > 0,
\end{equation} 
see Section \ref{po}.

The paper is organized as follows.   In Section \ref{aeconv} we prove an almost everywhere convergence result for the gradients.   In Section \ref{approx} we solve a sequence of approximate problems.   In Section \ref{main} we state and prove our main result.   Section \ref{po} is devoted to the proof of \eqref{pi} and to remainder terms for the Poincar\'e inequality.    In Section \ref{appl} we apply the main result to the potentials mentioned above.

\section{Almost everywhere convergence of the gradients} \label{aeconv}
  
Let us recall a classical result (see \cite{8}).

\begin{lem} \label{lem2.1}  For every $1<p<2$ there exists $c>0$ such that for all $x,y\in\mathbb{R}^N$,
\[
c\, |x-y|^2\big/\left(|x|+|y|\right)^{p-2} \leq \left(|x|^{p-2} x-|y|^{p-2}y\right) \cdot (x-y).
\]
For every $p\geq 2$ there exists $c>0$ such that for all $x,y\in \mathbb{R}^N$,
\[
c\, |x-y|^p\leq\left(|x|^{p-2} x-|y|^{p-2} y\right) \cdot (x-y).
\]
\end{lem}

We define the truncation $T$ by 
\[
Ts:=s, \ |s|\leq 1; \quad Ts:=s/|s| , \ |s|>1.
\]
The following theorem is a variant of a result which may be found e.g.\ in \cite{2}, \cite{5}, \cite{rak}. Here we provide a very simple argument, similar to that in \cite{5}.

\begin{theo} \label{theo2.2} Let $1<q\leq p$ and $(u_n) \subset W^{1,p}_{\textrm{loc}} (\Om)$ be such that for all $\omega \subset \subset \ \Om$
\begin{enumerate}
\item[(a)]   $\dis     \sup_n \ \|u_n\|_{W^{1,q}(\omega)} < +\infty$,
\item[(b)]   $\dis      \lim_{m,n \to\infty} \int_\omega \bigl(|\nabla u_m|^{p-2} \nabla u_m -\nabla u_n|^{p-2} \nabla u_n\bigr) \cdot \nabla T(u_m-u_n) \,dx = 0$.
\end{enumerate}
Then there exists a subsequence $(n_k)$ and $u\in W^{1,q}_{\textrm{loc}} (\Om)$ such that 
\begin{equation} \label{eq2.1}
u_{n_k} \to u\quad \text{and} \quad \nabla u_{n_k} \to \nabla u \quad \text{almost everywhere on }\Om.
\end{equation}
\end{theo}

\begin{proof}  It suffices to prove that for all $\omega \subset \subset  \Om$ there exist a subsequence $(n_k)$ and $u\in W^{1,q}(\omega)$ satisfying \eqref{eq2.1} a.e. on $\omega$ and to use a diagonal argument.

Let $\omega \subset \subset  \Om$. By assumption (a), extracting if necessary a subsequence, we can assume that for some $u\in W^{1,q}(\omega)$,
\[
u_n \to u \   \textrm{in} \ L^q (\omega),\quad u_n \to u  \  \textrm{a.e.} \ \textrm{on} \ \omega.
\]
Let us define
\[
\begin{aligned}
E_{m,n} & :=  \{x \in \omega : |u_m(x)-u_n(x)| < 1 \},
\\
e_{m,n} & :=  \left(|\nabla u_m|^{p-2} \nabla u_m - |\nabla u_n|^{p-2} \nabla u_n\right) \cdot \nabla (u_m-u_n).
\end{aligned}
\]
Then $e_{m,n}\ge 0$ and by assumption (b),
\[
\lim_{m,n\to\infty} \int_\omega e_{m,n} \chi_{E_{m,n}} \,dx = 0.
\]
Since $\chi_{E_{m,n}}\to 1$ a.e. on $\omega$ as $m,n\to\infty$, it follows, extracting if necessary a subfamily, that $e_{m,n}\to 0$ a.e. on $\omega$ as $m,n\to\infty$.   By Lemma \ref{lem2.1}, $|\nabla u_m -\nabla u_n|\to 0$ a.e. on $\omega$ as $m,n\to\infty$.   Hence $\nabla u_n \to v$ a.e. on $\omega$.    Since by assumption (a),
\[
\sup_n \|\nabla u_n\|_{L^q (\omega,\mathbb{R}^N)} < \infty,
\]
it follows from Proposition 5.4.7 in \cite{11} that 
\[
\nabla u_n  \rightharpoonup v \ \textrm{in} \ L^q (\omega,\mathbb{R}^N).
\]
We conclude that $v=\nabla u$. 
\end{proof}

\section{Approximate problems} \label{approx}

In this section we assume that \eqref{eq1.1} is satisfied, $V\in L^1_{\textrm{loc}}(\Omega)$ and $f\in W^{-1,p'} (\Om)$.   Let $0<\varepsilon < 1$.   We shall apply Ekeland's variational principle to the functional
\[
\varphi(u) := {1\over p} \int_\Om |\nabla u|^p \,dx -{1-\varepsilon\over p} \int_\Om V|u|^p \,dx + {\varepsilon\over p} \int_\Om |u|^p \,dx -\langle f,u\rangle
\]
defined on $W^{1,p}_0 (\Om)$.   By \eqref{eq1.1}, for every $u\in W^{1,p}_0 (\Om)$, $V|u|^p \in L^1 (\Om)$ and 
\[
\int_\Om V|u|^p \,dx \leq \int_\Om |\nabla u|^p \,dx.
\]
Moreover, the functional
\[
u \mapsto \int_\Om V|u|^p \,dx
\]
is continuous and G\^ateaux-differentiable on $W^{1,p}_0 (\Om)$ (see \cite{11}, Theorem 5.4.1).    It is clear that, on $W^{1,p}_0(\Om)$,
\begin{equation} \label{eq3.1}
-\|f\|_{W^{-1,p'}(\Om)} \|u\|_{W^{1,p}(\Om)} + {\varepsilon\over p} \|u\|^p_{W^{1,p}(\Om)} \leq \varphi(u).
\end{equation}
Hence $\varphi$ is bounded below and by Ekeland's variational principle (\cite{wi}, Theorem 2.4) there exists a sequence $(u_n)\subset W^{1,p}_0 (\Om)$ such that
\begin{equation} \label{eq3.2}
\varphi (u_n) \to c := \displaystyle \inf_{W^{1,p}_0(\Om)} \varphi \quad\text{and}\quad \varphi' (u_n)\to 0 \ \textrm{in} \ W^{-1,p'}(\Om).
\end{equation}
We deduce from \eqref{eq3.1} that
\begin{equation} \label{eq3.3}
\sup_n \|u_n\|_{W^{1,p}(\Om)} < +\infty.
\end{equation}
Going if necessary to a subsequence, we can assume the existence of $u\in W^{1,p}_0(\Om)$ such that 
\begin{equation} \label{eq3.4}
u_n \to u \quad \text{a.e. on } \Om.
\end{equation}

\begin{lem} \label{lem3.1}  Let $\zeta \in \mathcal{D}(\Om)$.   Then
\[
\lim_{m,n\to\infty} \int_\Om \bigl[|\nabla u_m|^{p-2} \nabla u_m - |\nabla u_n|^{p-2} \nabla u_n\bigr] \cdot \zeta \nabla T(u_m-u_n) \,dx = 0.
\]
\end{lem}

\begin{proof} Because of \eqref{eq3.2}, we have that 
\begin{equation} \label{eq3.5}
-\Delta_p u_n - (1-\varepsilon) V |u_n|^{p- 2} u_n + \varepsilon |u_n|^{p-   2} u_n = f+g_n,
\end{equation}
where $g_n \to 0$ in $W^{-1,p'}(\Om)$.       Testing \eqref{eq3.5} with $\zeta T(u_n-u_m)$, we see that it suffices to prove that
\begin{align}
&\lim_{m,n\to\infty} \int_\Om |\nabla u_m|^{p-1} |\nabla \zeta| \ |T(u_m-u_n)| \,dx = 0, \label{a1}
\\
&\lim_{m,n\to\infty} \int_\Om | u_m|^{p-1} | \zeta| \ |T(u_m-u_n)| \,dx = 0, \label{a2}
\\
&\lim_{m,n\to\infty} \int_\Om V| u_m|^{p-1} | \zeta| \ |T(u_m-u_n)| \,dx = 0. \label{a3}
\end{align}
By \eqref{eq3.4} and the fact that $|T(u_n-u_m)|\le 1$, $\lim_{m,n\to\infty}\int_{\textrm{spt}\, \zeta}|T(u_n-u_m)|\,dx = 0$. Hence 
\[
\lim_{m,n\to\infty} \int_{\textrm{spt}\, \zeta}  |T(u_m-u_n)|^{p} \,dx = 0,  \quad \lim_{m,n\to\infty} \int_{\textrm{spt}\,\zeta}  V|T(u_m-u_n)|^p \,dx = 0
\]
and \eqref{a1}, \eqref{a2} follow from \eqref{eq3.3} and H\"older's inequality. Since
\[
\int_{\textrm{spt}\, \zeta}V|u_m|^{p-1}|T(u_m-u_n)| \,dx \le \left(\int_{\textrm{spt}\, \zeta}V|u_m|^p\,dx\right)^{(p-1)/p} \left(\int_{\textrm{spt}\, \zeta}V|T(u_m-u_n)|^p\,dx\right)^{1/p}, 
\]
using \eqref{eq1.1} and \eqref{eq3.3} also \eqref{a3} follows.
\end{proof}

\begin{theo}\label{theo3.2}  There exists $u\in W^{1,p}_0(\Om)$ such that $\varphi(u)= \displaystyle \inf_{W^{1,p}_0(\Om)} \varphi$ and $\varphi'(u)=0$. 
\end{theo}

\begin{proof} Assumption (b) of Theorem \ref{theo2.2} (with $q=p$) follows from Lemma \ref{lem3.1}.     Extracting a subsequence, we can assume that 
\begin{equation} \label{eq3.6}
\nabla u_n \to \nabla u \quad \text{a.e. on } \Om.
\end{equation}
By \eqref{eq3.5} we have that, for every $\zeta \in \mathcal{D}(\Om)$, 
\[
\int_\Om |\nabla u_n|^{p-2} \nabla u_n \cdot \nabla \zeta \,dx-(1-\varepsilon) \int_\Om V|u_n|^{p-2} u_n \zeta \,dx + \varepsilon \int_\Om |u_n|^{p-2} u_n \zeta \,dx = \langle f+ g_n,\zeta\rangle.
\]
Using \eqref{eq3.3}, \eqref{eq3.4}, \eqref{eq3.6} and Proposition 5.4.7 in \cite{11}, we obtain  
\[
\int_\Om |\nabla u|^{p-2} \nabla u \cdot \nabla \zeta \,dx-(1-\varepsilon) \int_\Om V|u|^{p-2} u \zeta \,dx + \varepsilon \int_\Om |u|^{p-2} u \zeta \,dx = \langle f, \zeta\rangle,
\]
so that $\varphi'(u)=0$.   As in \cite{7}, the homogeneity of $\mathcal{Q}_V$ implies
\begin{align*}
\inf_{W^{1,p}_0 (\Om)} \varphi &=  \lim_{n\to\infty} \varphi (u_n) \\
&= \lim_{n\to\infty}  \left[{1\over p} \langle \varphi' (u_n),u_n \rangle +\left({1\over p}-1\right) \langle f,u_n\rangle\right] \\
&= \left({1\over p}-1\right)  \langle f,u\rangle 
= \varphi (u) - {1\over p} \langle \varphi'(u),u\rangle = \varphi (u).
\end{align*}
\end{proof}

\section{Main result} \label{main}

In this section we assume that (A) is satisfied and we define, on $\mathcal{D}'(\Om)$,
\[
\|f\| := \sup \{\langle f,u \rangle : u\in \mathcal{D}(\Om), \ \mathcal{Q}_V (u)=1\}
\]
so that
\begin{equation} \label{eq4.1}
\langle f,u \rangle \leq \|f\| \ [\mathcal{Q}_V(u)]^{1/p}.
\end{equation}
On the spaces
\[
X:=\{f\in \mathcal{D}'(\Om) : \|f\| < \infty \} \quad\text{and}\quad Y:=W^{1,q}_0 (\Om,Wdx)
\]
we respectively use the norm defined above and the natural norm. Note that the space $X$ has been introduced by Tak\'a\v{c} and Tintarev in \cite{9}.

\begin{lem} \label{lem4.1}   Let $f\in Y^\ast$. Then $f\in X$ and
\[
\|f\| \leq \|f\|_{Y^\ast}.
\]
\end{lem}

\begin{proof} Let $u\in \mathcal{D}(\Om)$.   By assumption \eqref{eq1.3} we have
\[
\langle f,u \rangle \leq \|f\|_{Y^\ast} \|u\|_Y \leq \|f\|_{Y^\ast} [\mathcal{Q}_V (u)]^{1/p}.
\]
\end{proof}

\begin{lem} \label{lem4.2}
\emph{(a)} Let $u\in W^{1,p}_0(\Om)$.   Then
\[
\|u\|_Y \leq \|u\|_{X^\ast} \leq [\mathcal{Q}_V (u)]^{1/p} \leq \|u\|_{W^{1,p}}(\Om).
\]
\emph{(b)} Let $f\in X$. Then $f\in W^{-1,p'}(\Om)$ and
\[
\|f\|_{W^{-1,p'}(\Om)}  \leq \|f\|.
\]
\end{lem}

\begin{proof} (a) Let $u\in \mathcal{D}(\Om)$. Using the Hahn-Banach theorem and the preceding lemma, we obtain
\[
\|u\|_Y = \sup_{f\in Y^\ast \atop \|f\|_{Y^\ast}\leq 1} \langle f,u \rangle  \leq  \sup_{f\in X \atop \|f\| \leq 1}    \langle f,u \rangle = \|u\|_{X^\ast}.
\]
It follows from \eqref{eq4.1} that
\[
\sup_{f\in X \atop \|f\|\leq 1} \langle f,u \rangle \leq [\mathcal{Q}_V (u)]^{1/p}.
\]
Since $V\geq 0$, it is clear that
\[
[\mathcal{Q}_V (u)]^{1/p} \leq \|u\|_{W^{1,p}(\Om)}.
\]
Now it is easy to conclude by density of $\mathcal{D}(\Om)$. 

(b) If $f\in X$ and $u\in \mathcal{D}(\Om)$, then
\[
\langle f,u \rangle \leq \|f\| [\mathcal{Q}_V(u)]^{1/p} \leq \|f\| \ \|u\|_{W^{1,p}(\Om)}. 
\]
\end{proof}

Let $f\in X$ and let $(\varepsilon_n)\subset \ ]0,1[$ be such that $\varepsilon_n \downarrow 0$.  Then $f\in W^{-1,p'}(\Omega)$, so  by Theorem \ref{theo3.2}, for every $n$ there exists $u_n\in W^{1,p}_0(\Om)$ such that
\begin{equation} \label{eq4.2}
-\Delta_p u_n - (1-\varepsilon_n)V |u_n|^{p-2} u_n + \varepsilon_n |u_n|^{p-2} u_n = f
\end{equation}
and $u_n$ minimizes the functional
\[
\varphi_n(v):={1\over p} \int_\Om |\nabla v|^p \,dx - {1-\varepsilon_n\over p} \int_\Om V|v|^p \,dx + {\varepsilon_n\over p} \int_\Om |v|^p  \,dx-\langle f,u \rangle.
\]
on $W^{1,p}_0 (\Om)$. In fact below we shall not use the minimizing property of $u_n$ but only the fact that \eqref{eq4.2} holds.

\begin{lem} \label{lem4.3}  Let $f\in X$.   Then
\[
\sup_n \|u_n\|_Y \leq \sup_n \mathcal{Q}_V (u_n) < \infty.
\]
\end{lem}

\begin{proof}  Lemma \ref{lem4.2} and equation \eqref{eq4.2} imply that
\begin{align*}
\|u_n\|^p_{X^\ast} & \leq \mathcal{Q}_V (u_n) 
\leq \mathcal{Q}_V (u_n) + \varepsilon_n \int_\Om (V+1) |u_n|^p \,dx
\\
&= \langle f,u_n \rangle
\leq \|f\|_X \|u_n\|_{X^\ast}.
\end{align*}
Since $p>1$, we obtain the conclusion. 
\end{proof}

Going if necessary to a subsequence, we can assume the existence of $u\in Y$ such that 
\begin{equation} \label{eq4.3}
u_n \to u \quad \text{a.e. on }\Om.
\end{equation}

\begin{lem} \label{lem4.4} Let $\zeta\in \mathcal{D}(\Om)$.   Then
\[
\lim_{m,n \to \infty}   \int_\Om \left[|\nabla u_m |^{p-2} u_m - |\nabla u_n|^{p-2}   u_n\right] \cdot \zeta \nabla T (u_m-u_n) \,dx = 0.
\]
\end{lem}

\begin{proof}   Because of \eqref{eq4.2}, as in the proof of Lemma \ref{lem3.1} it suffices to show that
\[
\begin{aligned}
&\lim_{m,n\to\infty} \int_\Om |\nabla u_m|^{p-1} |\nabla \zeta| \ |T(u_m-u_n)|\,dx=0,
\\
&\lim_{m,n\to\infty} \int_\Om | u_m|^{p-1} | \zeta| \ |T(u_m-u_n)|\,dx=0,
\\
&\lim_{m,n\to\infty} \int_\Om V| u_m|^{p-1} |\zeta| \ |T(u_m-u_n)|\,dx=0.
\end{aligned}
\]
We assume that $N>q$.    The other cases are similar but simpler.    By Lemma \ref{lem4.3}, $(u_n)$ is bounded in $W^{1,q}_{\textrm{loc}} (\Om)$, so by the  Sobolev theorem, $(u_n)$ is bounded in $L^{q^\ast}_{\textrm{loc}} (\Om)$.  Since by \eqref{eq4.3}, 
\[
\lim_{m,n\to\infty} \int_{\textrm{spt} \, \zeta}  |T(u_m-u_n)|^{({q\over p-1})'}   \,dx=0 \quad{and}\quad \lim_{m,n\to\infty} \int_{\textrm{spt}\ \zeta} V^r |T(u_m-u_n)|^r   \,dx=0,
\]
it is easy to conclude using \eqref{eq1.2} and H\"older's inequality. Note in particular that
\[
\int_{\textrm{spt}\, \zeta}V|u_m|^{p-1}|T(u_m-u_n)| \,dx \le \left(\int_{\textrm{spt}\, \zeta}|u_m|^{q^*}\,dx\right)^{(p-1)/q^*} \left(\int_{\textrm{spt}\, \zeta}V^r|T(u_m-u_n)|^r\,dx\right)^{1/r}.
\]
\end{proof}

\begin{theo} \label{theo4.5}   Assume  (A) is satisfied and $f\in \mathcal{D}' (\Om)$ is such that
\begin{equation} \label{cond}
\sup \{ \langle f,u \rangle : u\in \mathcal{D}(\Om),\ \mathcal{Q}_V (u)=1\} < +\infty.
\end{equation}
Then there exists $u\in W^{1,q}_0 (\Om,Wdx)$ such that, in $\mathcal{D}' (\Om)$,
\begin{equation} \label{eq}
-\Delta_p u-V(x) |u|^{p-2} u = f(x).
\end{equation}
\end{theo}

\begin{proof}    Let $\zeta \in \mathcal{D}(\Om)$.    By \eqref{eq4.2} we have 
\begin{equation} \label{eq4.4} 
\int_\Om |\nabla u_n|^{p-2} \nabla u_n \cdot \nabla \zeta \,dx-(1-\varepsilon_n) \int_\Om V |u_n|^{p-2} u_n \zeta \,dx + \varepsilon_n \int_\Om |u_n|^{p-2} u_n \zeta \,dx = \langle f,\zeta \rangle.
\end{equation}
Let us recall that $(u_n)$ is bounded in $W^{1,q}_{\textrm{loc}}(\Om)$ and $u_n\to u$ a.e. on $\Om$.    Assumption (b) of Theorem \ref{theo2.2} follows from Lemma \ref{lem4.4}.    Extracting a subsequence, we can assume that 
\[
\nabla u_n \to \nabla u \quad\text{a.e. on }\Om.
\]
We assume that $N>q$ and we choose $\omega$ such that
\[
\textrm{spt}\, \zeta \subset \omega \subset\subset \Om.
\]
Using Proposition 5.4.7 in \cite{11}, we obtain
\[
|\nabla u_n|^{p-2} \nabla u_n  \rightharpoonup |\nabla u|^{p-2} \nabla u \ \  \textrm{in} \ L^{{q\over p-1}} (\omega) \quad\text{and}\quad  
|u_n|^{p-2} u_n \rightharpoonup  |u|^{p-2}  u \ \ \textrm{in} \ L^{{q^\ast\over p-1}}(\omega).
\]
It follows then from \eqref{eq4.4} that
\[
\int_\Om |\nabla u|^{p-2} \nabla u \cdot \nabla \zeta \,dx - \int_\Om V|u|^{p-2} u \zeta \,dx = \langle f,\zeta \rangle. 
\]
\end{proof}

The following variant of Theorem \ref{theo4.5} will be needed in one of the applications in Section~\ref{appl}, see Theorem \ref{thmb}.

\begin{theo} \label{cor1} 
Theorem \ref{theo4.5} remains valid if we replace \eqref{eq1.2} (case $N>q$) in (A) by the conditions 
\begin{equation} \label{eqe}
\begin{aligned}
& V\in L^r_{\textrm{loc}}(\Omega) \quad \text{where }1/r+(p-1)(p-q)/q = 1 \\ 
 & \text{and} \quad \io V^{q/p}|u|^q\,dx \le C\io|\nabla u|^q\,dx  \quad \text{for some }C>0.
\end{aligned}
\end{equation}
\end{theo}

\begin{proof}
The argument is similar except that we must show that the third limit in the proof of Lemma \ref{lem4.4} is zero also when \eqref{eqe} holds and that we can pass to the limit in the second integral in \eqref{eq4.4}. Using Lemma \ref{lem4.3}, H\"older's inequality, \eqref{eqe} and the fact that $r=q/[(q-p+1)p]\ge 1$, we obtain
\begin{align*}
\lim_{m,n\to\infty} \int_{\textrm{spt} \, \zeta} V| u_m|^{p-1} |\zeta| & \, |T(u_m-u_n)|\,dx \le  \lim_{m,n\to\infty}\left(\int_{\textrm{spt} \, \zeta} V^{q/p}|u_m|^q\,dx\right)^{(p-1)/q} \times \\
& \times \left(\int_{\textrm{spt} \, \zeta} V^{r}|T(u_m-u_n)|^{pr}\,dx\right)^{1/(pr)} = 0.
\end{align*}
Let $E\subset \textrm{spt} \, \zeta$. Similarly as above, we have
\begin{align*}
\int_EV|u_n|^{p-1}\,dx  \le \left(\int_E V^{q/p}|u_n|^q\,dx\right)^{(p-1)/q} \left(\int_E V^{r}\,dx\right)^{1/(pr)} \le D \left(\int_E V^{r}\,dx\right)^{1/(pr)}.
\end{align*}
Since the integrand on the right-hand side is in $L^1(\textrm{spt} \, \zeta)$, it follows that $V|u_n|^{p-1}$ are uniformly integrable and we can pass to the limit in the second integral in \eqref{eq4.4} according to the Vitali theorem, see e.g. \cite[Theorem 3.1.9]{11}. 
\end{proof}

Note that in the case $q=p$ this result is stronger than Theorem \ref{theo4.5} because $V\in L^1_{\text{loc}}(\Omega)$ is allowed for any $p$.

\section{Poincar\'e inequality with remainder term} \label{po}

Let 
$\Omega := \omega\times \r^M$,   
where $\omega$ is a domain in $\r^{N-M}$ and $N>M>p$. For $x\in \Omega$ we shall write $x=(y,z)$, where $y\in\omega$ and $z\in \rm$.
Recall from the introduction that 
\[
\lambda_{1,p}(\Omega) := \inf\left\{\io|\nabla u|^p\,dx: u\in\calD(\Omega),\ \io|u|^p\,dx = 1\right\}.
\]
It is well known that $\lambda_{1,p}(\omega) = \lambda_{1,p}(\Omega)$ if $p=2$, see e.g. \cite{est}, Lemma 3. We shall show that this is also true for general $p\in(1,\infty)$.  

\begin{lem} \label{lem1}
$\lambda_{1,p}(\Omega) = \lambda_{1,p}(\omega)$.
\end{lem}

\begin{proof}
First we show that $\lambda_{1,p}(\Omega) \ge \lambda_{1,p}(\omega)$. Let $u\in\calD(\Omega)$, $\|u\|_{L^p(\Omega)}=1$. Then
\begin{align*}
\io|\nabla u|^p\,dx & \ge \io|\nabla_y u|^p\,dx = \irm dz\iom|\nabla_yu|^p\,dy \\
& \ge \irm dz \,\lambda_{1,p}(\omega)\iom|u|^p\,dy = \lambda_{1,p}(\omega)\io|u|^p\,dx = \lambda_{1,p}(\omega).
\end{align*}
Taking the infimum on the left-hand side we obtain the conclusion.

To show the reverse inequality, let $u(x)=v(y)w(z)$, where $v\in \calD(\omega)\setminus\{0\}$ and $w\in \calD(\r^M)\setminus\{0\}$. For each $\eps>0$ there exists $C_\eps>0$ such that 
\[
|\nabla u|^p \le (|w|\,|\nabla_yv|+|v|\,|\nabla_zw|)^p \le (1+\eps)|w|^p|\nabla_yv|^p + C_\eps |v|^p|\nabla_zw|^p. 
\]
Hence
\[
\lambda_{1,p}(\Omega)  \le \frac{\io|\nabla u|^p\,dx}{\io|u|^p\,dx} \le \frac{(1+\eps)\iom|\nabla_yv|^p\,dy}{\iom|v|^p\,dy} + \frac{C_\eps\irm|\nabla_zw|^p\,dz}{\irm|w|^p\,dz}.
\]
Taking the infimum with respect to $v$ and $w$, we obtain
\[
\lambda_{1,p}(\Omega)\le (1+\eps)\lambda_{1,p}(\omega).
\]
Since $\eps$ has been chosen arbitrarily, it follows that $\lambda_{1,p}(\Omega)\le \lambda_{1,p}(\omega)$.
\end{proof}

Now we state the main result of this section. 

\begin{theo} \label{poincare}
For each $u\in \calD(\Omega)$ the following holds: \\ 
\emph{(a)} If $p\ge 2$, then 
\[
\io(|\nabla u|^p-\lambda_{1,p}(\Omega)|u|^p)\,dx \ge \left(\frac{M-p}p\right)^p\io\frac{|u|^p}{|z|^p}\,dx.
\]
\emph{(b)} If $1<p<2$, then 
\[
\io(|\nabla u|^p-\lambda_{1,p}(\Omega)|u|^p)\,dx \ge 2^{(p-2)/2} \left(\frac{M-p}p\right)^p\io\frac{|u|^p}{|z|^p}\,dx.
\]
\end{theo}

\begin{proof} 
(a) Let $p\ge 2$. Then $(a+b)^{p/2} \ge a^{p/2}+b^{p/2}$ for all $a,b\ge 0$, hence
\[
|\nabla u|^p \ge |\nabla_yu|^p+|\nabla_zu|^p.
\]
Using this, Lemma \ref{lem1} and Hardy's inequality in $\rm$, we obtain
\begin{align*}
\io(|\nabla u|^p & -\lambda_{1,p}(\Omega)|u|^p)\,dx \\ 
& \ge \irm dz\iom(|\nabla_yu|^p-\lambda_{1,p}(\Omega) |u|^p)\,dy + \iom dy\irm|\nabla_zu|^p\,dz \\
& \ge \iom dy\irm|\nabla_zu|^p\,dz \ge \iom dy \left(\frac{M-p}p\right)^p\irm \frac{|u|^p}{|z|^p}\,dz \\
& = \left(\frac{M-p}p\right)^p\io\frac{|u|^p}{|z|^p}\,dx.
\end{align*}

(b) Let $1<p<2$. It is easy to see that $(a+b)^{p/2} \ge 2^{(p-2)/2} (a^{p/2}+b^{p/2})$ for such $p$ and all $a,b\ge 0$. Hence
\[
|\nabla u|^p \ge2^{(p-2)/2}\left( |\nabla_yu|^p+|\nabla_zu|^p\right)
\]
and we can proceed as above. 
\end{proof}

Here we have not excluded the case $ \lambda_{1,p}(\omega)=0$ but the result is only interesting if $\lambda_{1,p}(\omega)$ is positive. A sufficient condition for this is that $\omega$ has finite measure. 

\section{Applications} \label{appl}

In this section we work out some applications of Theorem \ref{theo4.5} for the potentials $V$ mentioned in the introduction. Let $\Omega$ be a domain in $\rn$. 
If $\Omega$ is bounded, $0\in\Omega$ and $V$ is the quadratic Hardy potential \eqref{hardy1}, then more can be said about the solution $u$ and the space $H^*$ in Theorem \ref{thma}. Given $1\le q<2$, we have
\[
\io |\nabla u|^2\,dx - \io V(x)u^2\,dx \ge C(q,\Omega)\|u\|_{W^{1,q}(\Omega)}^2
\] 
for all $u\in\calD(\Omega)$ (\cite{10}, Theorem 2.2). We see that $V\in L^r(\Omega)$ if $N/(N-1)<q<2$ ($r$ is as in \eqref{eq1.2}) and \eqref{eq1.3} holds with constant $W$. So $H^*=X$ and we also have $Y=W^{1,q}_0(\Omega)$ where $X,Y$ are as in Section \ref{main}. Hence by Lemma \ref{lem4.2} and Theorem \ref{theo4.5}, $H^*\subset W^{-1,2}(\Omega)$ and the solution $u$ is in $W^{1,q}_0(\Omega)$ for any  $1\le q<2$. 

Let $x=(y,z)\in\Omega\subset \r^k\times\r^{N-k}$, where $N\ge k>p>1$, and consider the Hardy potential
\[
V(x) := \left(\frac{k-p}p\right)^p|y|^{-p}.
\]
 
\begin{theo} \label{thmb}
Let $\Omega$ be a bounded domain containing the origin and let $V$ be the Hardy potential above. Then for each $f$ satisfying \eqref{cond} and each $1\le q<p$ there exists a solution $u\in W^{1,q}_0(\Omega)$ to \eqref{eq}. 
\end{theo}

Recall from Lemma \ref{lem4.2}  that if $f$ satisfies \eqref{cond}, then $f\in W^{-1,p'}(\Omega)$. 

\begin{proof}
According to Lemma 2.1 in \cite{3} (see also \cite{1}, Theorem 1.1), for each $1<q<p$ there exists a constant $C(q,\Omega)$ such that 
\[
\mathcal{Q}_V(u) \ge C(q,\Omega)\left(\io|\nabla u|^q\,dx\right)^{p/q}, \quad u\in\calD(\Omega). 
\]
So the Poincar\'e inequality implies that \eqref{eq1.3} holds for some constant $W$. Since also \eqref{eqe} holds if $q$ is sufficiently close to $p$, we obtain
the conclusion using Theorem \ref{cor1} ($k(p-1)/(k-1)<q<p$ is needed in order to have $V\in L^r_{\text{loc}}(\Omega)$ with $r$ as in \eqref{eqe}). 
\end{proof}

This result  extends Theorem 3.1 in \cite{3} where it was assumed that $f\in L^\gamma(\Omega)$ for some $\gamma>(p^*)'$. 

In our next theorem we essentially recover the main result (Theorem 4.3) of \cite{9}.

\begin{theo} \label{thmc}
Let $\Omega$ be a domain and let $V\in L^\infty_{\emph{loc}}(\Omega)$, $V\ge 0$. Suppose that 
\begin{equation} \label{tt}
\mathcal{Q}_V(u) \ge \io\wt W|u|^p\,dx  \quad \text{for all }  u\in\calD(\Omega) \text{ and some } \wt W\in \mathcal{C}(\Omega), \ \wt W>0.
\end{equation}
Then for each $f$ satisfying \eqref{cond} there exists a solution $u\in W^{1,p}_{\emph{loc}}(\Omega)$ to \eqref{eq}.
\end{theo}

\begin{proof}
According to Proposition 3.1 in \cite{9} (see also (2.6) there), \eqref{tt} implies that $\mathcal{Q}_V$ satisfies \eqref{eq1.3} with $q=p$. Hence our Theorem \ref{theo4.5} applies. 
\end{proof}

\begin{theo} \label{thmd}
Let $\Omega = \omega\times \r^M$,  where $\omega\subset \r^{N-M}$ is a domain such that $\lambda_{1,p}(\omega)>0$, $N>M>p$, and let $V(x) = \lambda_{1,p}(\Omega)$. Then for each $f$ satisfying \eqref{cond} there exists a solution $u\in W^{1,p}_{\emph{loc}}(\Omega)$ to \eqref{eq}.
 \end{theo}

\begin{proof}
Let $x=(y,z)\in\omega\times\r^M$.  It follows from Theorem \ref{poincare} that
\[
\mathcal{Q}_V(u) = \io|\nabla u|^p\,dx-\lambda_{1,p}(\Omega)\io|u|^p\,dx \ge \io \wt W|u|^p\,dx \quad \text{for all }u\in\calD(\Omega),
\]
where $\wt W(x)=C_p/(1+|z|^p)$ and $C_p$ is the constant on the right-hand side of respectively (a) and (b) of Theorem \ref{poincare}. So the conclusion follows from Theorem \ref{thmc}.
\end{proof}

Below we give an example showing that the solution we obtain need not be in $W^{1,p}(\Omega)$. 

\begin{example}
\emph{
Let $\Omega = \omega\times \rm$, where $\omega\subset\r^{N-M}$ is such that $\lambda_1 := \lambda_{1,2}(\omega) > 0$ and put
\[
\mathcal{Q}(u) := \io(|\nabla u|^2-\lambda_1u^2)\,dx.
\] 
By the definition of $\lambda_1$, for each $n$ we can find $u_n\in \calD(\Omega)$ such that
\[
\left(1-\frac1n\right)\io|\nabla u_n|^2\,dx \le \lambda_1\io u_n^2\,dx.
\]
Hence 
\[
Q(u_n) \le \frac1n\io |\nabla u_n|^2\,dx.
\]
By normalization, we can assume 
\[
Q(u_n) = \frac1{n^2}.
\]
Translating along $\rm$, we may assume $\text{spt\,}u_n\cap \text{spt\,}u_m =\emptyset$ if $n\ne m$. We define 
\[
f_n := -\Delta u_n-\lambda_1u_n \quad \text{and} \quad f:= \sum_{n=1}^\infty f_n \quad \text{in }\calD'(\Omega).
\]
Then $f\in X$ because 
\[
\|f\| = \left(\sum_{n=1}^\infty \frac1{n^2}\right)^{1/2} < \infty,
\]
and $u = \sum_{n=1}^\infty u_n$ is a weak solution for the equation
\[
-\Delta u -\lambda_1 u = f(x).
\]
Moreover, $u$ is the unique solution in $H$ as follows from Theorem \ref{thma}. But
\[
\io|\nabla u|^2\,dx = \sum_{n=1}^\infty\io|\nabla u_n|^2\,dx \ge \sum_{n=1}^\infty \frac1n = +\infty,
\] 
so $u\not\in W^{1,2}(\Omega)$. 
}
\end{example}


\begin{thebibliography}{99} 

\bibitem{1}  Abdellaoui, B., Colorado, E., Peral, I., ``Some improved Caffarelli-Kohn-Nirenberg inequa\-lities", \textit{Calc. Var. Partial Differential Equations} 23 (2005), 327-345.

\bibitem{2}  Boccardo, L., Murat, F., ``Almost everywhere convergence of the gradients of solutions to elliptic and parabolic equations", \textit{Nonlinear Anal.} 19 (1992), 581-597.

\bibitem{3}  Brandolini, B., Chiacchio, F., Trombetti, C., ``Some remarks on nonlinear elliptic problems involving Hardy potentials", \textit{Houston J. Math.} 33 (2007), 617-630.

\bibitem{4}  Brezis, H., V\'azquez, J., ``Blow-up solutions of some nonlinear elliptic problems", \textit{Rev. Mat. Univ. Complut. Madrid} 10 (1997), 443-469.

\bibitem{5}  de Valeriola, S., Willem, M., ``On some quasilinear critical problems", \textit{Adv. Nonlinear Stud.} 9 (2009), 825-836.

\bibitem{dem}  del Pino, M., Elgueta, M., Man\'asevich, R., ``A homotopic deformation along $p$ of a Leray-Schauder degree result and existence for $(|u'|^{p-2}u')'+f(t,u)=0$, $u(0)=u(T)=0$, $p>1$, \emph{J. Diff. Eq.} 80 (1989), 1-13.

\bibitem{6}  Dupaigne, L., ``A nonlinear elliptic PDE with the inverse square potential", \textit{J. Anal. Math.} 86 (2002), 359-398.

\bibitem{est} Esteban, M.J., ``Nonlinear elliptic problems in strip-like domains: symmetry of positive vortex rings'', \textit{Nonl. Anal.} 7 (1983), 365-379.

\bibitem{fhtt} Fleckinger-Pell\'e, J., Hern\'andez, J., Tak\'a{\v c}, P., de Th\'elin, F., ``Uniqueness and positivity for solutions of equations with the $p$-Laplacian'', \emph{Lecture Notes in Pure and Appl. Math.} 194, Dekker, New York, 1998, pp.\ 141--155.

\bibitem{7}  Garcia Azorero J.P., Peral Alonso, I.,  ``Hardy inequalities and some critical elliptic and parabolic pro\-blems", \textit{J. Diff. Eq.} 144 (1998), 441-476.

\bibitem{rak} Rakotoson, J.M., ``Quasilinear elliptic problems with measures as data'', \textit{Diff. Int. Eq.} 4 (1991), 449-457.

\bibitem{8}  Simon, J., ``R\'egularit\'e de la solution d'une \'equation non lin\'eaire dans $\mathbb{R}^N$, \textit{Lecture Notes in Math.} 665, Springer, Berlin, 1978, 205-227.

\bibitem{9}  Tak\'a\v{c}, P., Tintarev, K., ``Generalized minimizer solutions for equations with the $p$-Laplacian and a potential term", \textit{Proc. Roy. Soc. Edinburgh Sect. A} 138 (2008), 201-221.

\bibitem{10}  Vazquez, J., Zuazua, E., ``The Hardy inequality and the asymptotic behaviour of the heat equation with an inverse-square potential", \textit{J. Funct. Anal.} 173 (2000), 103-153.

\bibitem{wi} Willem, M., ``Minimax Theorems'', Birkh\"auser, Boston, 1996.

\bibitem{11}  Willem, M., ``Functional Analysis. Fundamentals and Applications", Birkh\"auser/Springer, New York, 2013.

 

\end{thebibliography}
\end{document}